\documentclass[12pt,a4paper]{amsart}

\usepackage{color}
\usepackage{amsmath}
\usepackage{amssymb}
\usepackage{amsthm}
\usepackage{amsfonts}
\usepackage{latexsym}
\usepackage{mathrsfs}
\usepackage{amsfonts}
\usepackage{graphicx}
\usepackage{graphpap}
\usepackage{tikz}
\usepackage{array}

\theoremstyle{plain}
\newtheorem{theorem}{Theorem}[section]

\newtheorem{proposition}[theorem]{Proposition}
\theoremstyle{definition}
\newtheorem{definition}[theorem]{Definition}
\newtheorem{remark}{Remark}
\numberwithin{equation}{section}

\newtheorem{example}{Example}
\title[Primitive Central idempotents of a Rational Abelian Group Algebra]{Primitive Central idempotents and the Wedderburn Decomposition of a Rational Abelian Group Algebra}
\author[Ravi S. Kulkarni and Soham S. Pradhan]{Ravi S. Kulkarni and Soham S. Pradhan}

\begin{document}
\maketitle
\author
\begin{abstract}
In this paper we give an explicit description of primitive central idempotents of rational group algebras of finite abelian groups using {\it{long presentation}}, and determine their Wedderburn decompositions.\\\
\end{abstract}
\section{Introduction}
The classical approach of computing primitive central
idempotents of a rational abelian group algebra $\mathbb{Q}[G]$, is to first compute the primitive central idempotents
$$e(\chi) = \frac{\chi(1)}{\mid{G}\mid}{\sum_{g \in G} {\chi(g^{-1})g}}$$ of $\mathbb{C}[G]$
associated with complex irreducible characters $\chi$, and then sum the primitive central idempotents $e(\sigma o \chi)$ with
$\sigma \in $ Gal$({\mathbb{Q}(\chi)}/{\mathbb{Q}})$ to obtain the associated rational primitive central idempotent
$$e_{\mathbb{Q}}(\chi)=\displaystyle\sum_{\sigma \in Gal({\mathbb{Q}(\chi)}/{\mathbb{Q}})}{\sigma(e(\chi))}.$$
An explicit description of primitive central idempotents of
rational group algebra of a finite abelian group using subgroups structure was considered by several authors (see Ayoub \cite{Ay}, Theorem 5, Milies\cite{Mi}, Theorem 1.4, Jespers\cite{Je}, Theorem 1.3). In this paper, we aim to give an another explicit description of primitive central idempotents rational group algebra of a finite abelian group using {\it{long presentation}}.

It is a known fact that every primitive central idempotent of rational group algebra of a finite abelian group is product of primitive central idempotents of its $p$-primary parts, $p$ a prime. This result brings our attention on the problem of computing primitive central idempotents of  rational group algebra of a finite abelian $p$-group. In this paper, we approach to this problem in the following way.

Let $G$ be an abelian $p$-group. Then $G$ has a subnormal series
$$\langle{e}\rangle = G_{o} < G_{1} <  \dots  < G_{n} = G$$
such that each factor group $G_{i}/G_{i - 1}$ is a cyclic group of order $p$. Existence of this composition series allows us to define {\it long presentation} (see section 2) of $G$, and moreover, there is an {\it{enumeration}} (see section 2) in the system of long generators.
We shall inductively compute the complete set of primitive central idempotents of $\mathbb{Q}[G]$ using that {\it{long presentation}}.

A remarkable thing is that as a consequence of this inductive process, every primitive central idempotents of $\mathbb{Q}[G]$ can be expressed in  product form. Moreover, for a cyclic group of order $p^{n}$, $n$ a positive integer, in terms of characters the expression of each primitive central idempotent of
complex group algebra has $p^n$ terms, and such a primitive central idempotent factors into $n$ factors, each factor containing $p$ terms.
So each primitive central idempotent has an expression containing $pn$ terms, and therefore over $\mathbb{Q}$ this number is less than or equal to $pn$.

It is an well known fact that Wedderburn decomposition of rational group algebra of a finite abelian group is direct sum of cyclotomic fields. The Classical theory says that the coefficient of cylotomic field $\mathbb{Q}(\zeta_{p^{r}})$, where $\zeta_{p^{r}}$ is a primitive $p$th root of unity, appearing in the Wedderburn decomposition of $\mathbb{Q}[G]$ is equal to the number of subgroups of $G$ with factor groups isomorphic to cyclic group of order ${p^r}$. We shall explicitly compute those coefficients.

We briefly describe the organization of the paper. In section 2, we shall define {\it{short and long presentations}} of an abelian $p$-group.
In section 3, we shall define PCI-diagrams of an abelian $p$-group. In section 4, we shall give an explicit expression of primitive central idempotents of an abelian $p$-group over an algebraically closed field using {\it{long presentation}}.
In section 5, we shall give an explicit expression of primitive central idempotents of rational group algebra of a finite abelian $p$-group using {\it{long presentation}}.
In section 6, we shall compute coefficients of cyclotomic fields appearing in the Wedderburn decomposition of rational group algebra of a finite abelian $p$-group.

We use the following notations. Throughout this paper we assume that $G$ is a finite group. We denote the order of $G$ by $|G|$.
By $F$ we mean a field with characteristic does not divide $|G|$.
By $F[G]$, we mean the group algebra of $G$ over $F$. For $X \subset G$, $\langle{X}\rangle$
denotes the subgroup generated by $X$.
For $H \leqq G$, $\widehat{H}$ denotes the idempotent $\frac{1}{\mid{H}\mid}\sum_{h \in H}h$ of $\mathbb{Q}[G]$.
We denote cyclic group of order $n$ by $C_{n}$.
 For a positive integer $n$, $\zeta_{n}$ denotes a complex primitive $n$th root of unity and $\mathbb{Q}(\zeta_{n})$ denotes $n$th cyclotomic field over $\mathbb{Q}$. By $p$ we always mean a prime number. Throughout this paper we use $e_{X}$ for $\frac{1 + X + \dots + X^{p - 1}}{p}$, where $X$ is an indeterminate, and  $e^{'}_{X}$ for $1 - e_{X}$.
\section{Short and Long Presentations of an Abelian $p$-Group}
Let $G$ be an abelian $p$-group of order $p^N$. It is known that abelian groups of order $p^N$ are
parametrized by partitions of $N$. Let $N = s_1 + s_2+ ... + s_m$ be a partition of $N$, and let $r_i$ is a divisor of $s_i$, so that after re-indexing $r_1 > r_2 > \dots > r_m \geq 1$, and for each $i$, let $s_i = r_il_i$.
To this partition of $N$, we associate the abelian $p$-group
\begin{equation}
G = {\prod_{i = 1}^{m}\prod_{j = 1}^{l_i}{C_{p^{r_{i}}, j}}},
\end{equation}
where $C_{p^{r_i}, j}$ denotes the $j$th factor of the {\it{homo-cyclic component}} of exponent ${p^{r_i}}$.  We define a {\it{short presentation }} of $C_{p^{r_i}, j}$ as $\langle{ y_{(r_{i},j)}| y^{p^{r_i}}_{(r_{i},j)} = e}\rangle$, and the generator $y_{(r_{i},j)}$ is called {\it{short generator}} of $C_{p^{r_i}, j}$. Therefore a {\it{short presentation }} of $G$ is defined by
\begin{align*}
\prod_{i = 1}^{m}\prod_{j = 1}^{l_i}\langle{y_{(r_{i},j)} ~|~ y^{p^{r_i}}_{(r_{i},j)} = e}\rangle.
\end{align*}
Notice that a short presentation contains minimum number of generators. Now we make a refinement of {\it{short presentation}}, and use the terminology {\it{long presentation}} for that. We use three in-dices to represent a generator in a {\it{long presentation}} of $G$. We define a {\it{long presentation}} of $C_{p^{r_i}, j}$ as
\begin{align*}
\langle x_{\{(r_{i},j),r_{i}\}}, x_{\{(r_{i},j),r_{i} - 1\}}, ... , x_{\{(r_{i},j), 1\}} |
& x^{p}_{\{(r_{i},j),r_{i}\}} =x_{\{(r_{i},j),r_{i} - 1\}},\\
& x^{p}_{\{(r_{i},j),r_{i} - 1\}} = x_{\{(r_{i},j),r_{i} - 2\}},\\
& \dots ,\\
& x^{p}_{\{(r_{i},j),1\}} = 1 \rangle.
\end{align*}
We call the set $\{x_{\{(r_{i},j),r_{i}\}}, x_{\{(r_{i},j),r_{i} - 1\}}, ... , x_{\{(r_{i},j), 1\}}\}$, a {\it{system of long generators}} of $C_{p^{r_i}, j}$.
Therefore a {\it{long presentation}} of $G$ is defined by
\begin{align*}
\prod_{i = 1}^{m}\prod_{j = 1}^{l_i}\langle x_{\{(r_{i},j),r_{i}\}}, x_{\{(r_{i},j),r_{i} - 1\}}, ... , x_{\{(r_{i},j), 1\}} |
& x^{p}_{\{(r_{i},j),r_{i}\}} = x_{\{(r_{i},j),r_{i} - 1\}},\\
& x^{p}_{\{(r_{i},j),r_{i} - 1\}} = x_{\{(r_{i},j),r_{i} - 2\}},\\
&\dots ,\\
&x^{p}_{\{(r_{i},j),1\}} = 1, \\
& com \rangle.
\end{align*}
We call $(s,j)$, the place index and $a$, the power index of the long generator $x_{\{(s,j), a\}}$. An interesting thing is that, there is an enumeration in a system of long generators with respect to the lexicographic order. Note that, every element of $G$ can be expressed uniquely as product of $x^{\alpha_{\{{(s,j), a}\}}}_{\{(s,j), a\}}$'s, where
$0 \leq \alpha_{\{{(s,j), a}\}} \leq p - 1$.

\section{PCI-Diagram of an Abelian $p$-Group}
In this section we define PCI-diagram of an abelian $p$-group.
\begin{definition}
Let $G$ be an abelian $p$-group of order $p^N$. Then $G$ has a subnormal series
$$\{e\} = G_{o} < G_{1} < \dots < G_{N} = G$$ such that
$G_{i}/ G_{i - 1} = \langle{x_{i}{G_{i - 1}}}\rangle$, for some $x_{i}$ in $G_{i}$  and is isomorphic to $C_{p}$.
Let $p$ be an invertible element in $F$. Let $I_{G_{i}}$ denotes the complete set of primitive central idempotents of $F[G_{i}]$,
where $i$ runs over the set  $\{0, 1, \dots , N\}$. Let $e$ be an element in $I_{G_{i}}$, and $(\eta, W)$ be its associated irreducible $F$-representation. Suppose that the induced representation Ind$(\eta, W)\uparrow_{G_{i-1}}^{G_{i}}$ decomposes into $l$ distinct irreducible
$F$-representations $(\rho_{1}, V_{1})$, $(\rho_{2}, V_{2})$, $\dots$ , $(\rho_{l}, V_{l})$ with multiplicities $n_{1}$, $n_{2}$, $\dots$, $n_{l}$ respectively, i.e.,
\begin{center}
 Ind$(\eta, W)\uparrow_{G_{i-1}}^{G_{i}} = \displaystyle{\bigoplus_{j = 1}^{l}{n_{j}}(\rho_{j}, V_{j})}$.
\end{center}
For each $j$, let $e_{j}$ be the primitive central idempotent associated to the irreducible representation  $(\rho_{j}, V_{j})$.
We define the PCI-diagram associated to the above composition series is a graph, whose vertex set is $\cup_{i = 0}^{n}{I_{G_{i}}}$, and the vertex $e$ is adjacent to $l$ vertices $e_{1}, e_{2}, \dots , e_{l}$. We call $I_{G_{i}}$ as vertex set at the $i$th level of the PCI-diagram.
\end{definition}
\begin{remark}
The definition of PCI-diagram can be generalized for finite solvable groups, as a finite solvable group always has a subnormal series
with successive quotient groups are cyclic groups of prime order.
\end{remark}
\section{Primitive Central Idempotents of an Abelian Group Over an Algebraically Closed Field:}
In this section, we compute primitive central idempotents of an abelian group over an algebraically closed field.
But to compute primitive central idempotents of an abelian group over an algebraically closed field it is sufficient to compute primitive central idempotents of a cyclic group of prime power order. In this case, we shall see that the primitive central idempotents factor nicely.
\begin{theorem}
Let $G$ be $C_{p^{n}}$, with the long presentation:
\begin{center}
 $G = \langle{x_{1}, x_{2}, \dots ,x_{n} \mid x_{1}^{p} = 1, x_{2}^{p} = x_{1}, \dots , x_{n}^{p} = x_{n - 1}}\rangle$.
\end{center}
Let $F$ be an algebraically closed field with characteristic 0. Let $\zeta_{n}$ be a $p^{n}$th root of unity in $F$.
Then every primitive central idempotent of ${F}[G]$ is of the form:
\begin{center}
$e_{\zeta_{1}x_{1}}e_{\zeta_{2}x_{2}} \dots e_{\zeta_{n}x_{n}}$,
\end {center}
where $\zeta_{i}$'s are $p^i$-th power roots of unity in $F$ defined by $\zeta^{p}_{n} = \zeta_{n - 1}, \dots, \zeta^{p}_{2} = \zeta_{1}$.

\end{theorem}
\begin{proof}
 Every element of $G$ can be expressed uniquely as $x_{1}^{i _1}x_{2}^{i_2} \dots x_{n}^{i_n}$ where
 $i_{j}$'s are running over the set $\{0, 1, \dots , p -1\}$. Let $\rho$ be an $F$- irreducible representation of $G$,
 then $\rho(x_j)$ is a $p^j$-th root of unity, $\zeta_{j}$(say) and then $\rho(x^{p}_{j})= \{\rho(x_{j})\}^{p} = \zeta_{j}^p = \zeta_{j - 1}$, $j = 1,2, \dots , n$.
 Then the primitive central idempotent corresponding to $\rho$ is:
\begin{center}
 $$e_{\rho} = \displaystyle\frac{1}{p^n}{\sum_{g \in G}\rho(g^{-1})g}$$
 $$ = \displaystyle\frac{1}{p^n}\{{\sum_{i_{1} = 0}^{ p - 1}\sum_{i_{2} = 0}^{ p - 1} \dots \sum_{i_{n - 1} = 0}^{ p - 1}
 \sum_{i_{n} = 0}^{ p - 1}\rho{\{({x_{1}^{i _1}x_{2}^{i_2} \dots x_{n}^{i_n}})^{-1}\}}(x_{1}^{i _1}x_{2}^{i_2} \dots x_{n}^{i_n}})\}$$
\end{center}
This implies that
\begin{center}
 $$e_{\rho}=\frac{1}{p^{n - 1}} \{{\sum_{i_{1} = 0}^{ p - 1}\sum_{i_{2} = 0}^{ p - 1} \dots \sum_{i_{n - 1} = 0}^{ p - 1}
             \rho({x_{1}^{i _1}x_{2}^{i_2} \dots x_{n - 1}^{i_{n - 1}}})^{-1}(x_{1}^{i _1}x_{2}^{i_2} \dots x_{n - 1}^{i_{n - 1}}})\}$$
             
            $$\{\frac{1}{p}(1 + \frac{x_n}{\rho(x_n)}+ \dots + \frac{x_n^{p - 1}}{\rho(x_n)^{p - 1}})\}$$

$$= \frac{1}{p^{n - 1}} \{{\sum_{i_{1} = 0}^{ p - 1}\sum_{i_{2} = 0}^{ p - 1} \dots \sum_{i_{n - 1} = 0}^{ p - 1}
             \rho({x_{1}^{i _1}x_{2}^{i_2} \dots x_{n - 1}^{i_{n - 1}}})^{-1}(x_{1}^{i _1}x_{2}^{i_2} \dots x_{n - 1}^{i_{n - 1}}})\}
             e_{\zeta^{-1}_{n}{x_{n}}}.$$

\end{center}
Similarly, if we keep continuing this process, finally we get:
$$e_{\rho} = e_{{\zeta^{-1}_1}{x_{1}}}e_{{\zeta^{-1}_2}{x_{2}}} \dots e_{{\zeta^{-1}_n}{x_{n}}},$$
where $\zeta^{p}_{n} = \zeta_{n - 1}, \dots, \zeta^{p}_{2} = \zeta_{1}$.
Replacing $x_{i}$ by $x^{-1}_{i}$ we get:
$$e_{\rho} = e_{{\zeta^{-1}_1}{x^{-1}_{1}}}e_{{\zeta^{-1}_2}{x^{-1}_{2}}} \dots e_{{\zeta^{-1}_n}{x^{-1}_{n}}}.$$
This completes the proof of the theorem.
\end{proof}
\begin{theorem}
Let $G$ be $C_{p^{n}}$, with the long presentation:
\begin{center}
 $G = \langle{x_{1}, x_{2}, \dots ,x_{n} \mid x_{1}^{p} = 1, x_{2}^{p} = x_{1}, \dots , x_{n}^{p} = x_{n - 1}}\rangle$.
\end{center}
Let $F$ be an algebraically closed field with characteristic 0.
Let $H$ be the unique subgroup of index $p$ of $G$ with the long presentation:
\begin{center}
$H = \langle{x_{1}, x_{2}, \dots ,x_{n - 1} \mid x_{1}^{p} = 1, x_{2}^{p} = x_{1}, \dots , x_{n - 1}^{p} = x_{n - 2}}\rangle$.
\end{center}
Let $\eta$ be an $F$- irreducible representation of $H$ and let $e_{\eta}$ be its corresponding primitive central idempotent. Let
$e_{\eta} = e_{\zeta_{1}x_{1}}e_{\zeta_{2}x_{2}}\dots e_{\zeta_{n - 1}x_{n - 1}}$ with
$\zeta^{p}_{n - 1} = \zeta_{n - 2}, \dots, \zeta^{p}_{2} = \zeta_{1}$. Then:
\begin{enumerate}
 \item[$(1)$] $\eta$ extends to $p$ mutually inequivalent $F$-irreducible representations $\rho_{0}, \rho_{1}, \dots , \rho_{p - 1}$
 of $G$.
 \item[$(2)$] The primitive central idempotents associated to $\rho_{i}$'s are $e_{\eta}e_{\epsilon^{i}\zeta_{n}x_{n}}$,
where $i$ runs over the set $\{0,1, \dots, p - 1\}$, $\zeta_{n}$ is a fixed $p$-th root of $\zeta_{n - 1}$ and $\epsilon$ is a primitive
$p$-th root of unity in $F$.
\end{enumerate}
\end{theorem}
\begin{proof}
Since every $F$- irreducible representation of $G$ is $1$- dimensional, then clearly $\eta$ extends to an
$F$- irreducible representation of $G$. Suppose that $\eta$ extends to $\rho$, this implies that $\{\rho(x_{n})\}^{p} = \rho({x_{n}}^p) = \eta({x_{n}}^p) = \eta(x_{n - 1}) =
\zeta_{n - 1}$. Therefore $\rho(x_{n})$ is equal to $\epsilon^{i}{\zeta_{n}}$, where $\epsilon$
is a primitive $p$th root of unity, and $\zeta_{n}$ is a fixed $p$-th root of $\zeta_{n - 1}$, and $i$ runs over the set $\{0, 1, \dots , p - 1\}$.
So $\eta$ extends precisely $p$ distinct ways. This completes the first part of the theorem.\\
Let $\rho_{0}, \rho_{1}, \dots , \rho_{p - 1}$ be are $p$ extensions of $\eta$, and
$e_{\rho_{0}}, e_{\rho_{1}}, \dots , e_{\rho_{p - 1}}$ be their corresponding primitive central idempotents of $F[G]$.
Then for each $i \in\{0, 1, \dots , p -1\}$,
\begin{center}
 $e_{\rho_{i}} = \displaystyle{\frac{1}{p^n}\sum _{g \in G}{\rho_{i}(g^{-1})g}}$.
\end{center}
Again the previous expression can be written as:
\begin{center}
\begin{tiny}
 $e_{\rho_{i}}  =  \displaystyle{\frac{1}{p^n}[\sum _{h \in H}{\rho_{i}(h^{-1})h} +
 \sum _{h \in H}{\{{\rho_{i}(h^{-1})h\}}{\{\rho_{i}(x_{n}^{-1}){x_{n}}}\}}
  + \dots + \sum _{h \in H}{\{{\rho_{i}(h^{-1})h\}}{\{\rho_{i}(x_{n}^{-(p - 1)}){x_{n}^{p - 1}}}\}}]}$
  \end{tiny}
\end{center}
Therefore,
\begin{center}
 $e_{\rho_{i}} = \displaystyle{\{\frac{1}{p^{n}}\sum_{h \in H}\rho_{i}(h^{-1})h\}\{1 + \frac{x_n}{\rho_i(x_n)} + \dots +
 \frac{x_n^{p - 1}}{\rho_i(x_n^{p - 1})}\}}$
\end{center}
\begin{center}
 $= \displaystyle{\frac{e_{\eta}}{p}\{1 + \frac{x_n}{\rho_i(x_n)} + \dots + \frac{x_n^{p - 1}}{\rho_i{(x_n^{p - 1})}}\}}$
\end{center}
\begin{center}
 $= \displaystyle{\frac{e_{\eta}}{p}\{1 + \frac{x_n}{\rho_i(x_n)} + \dots + \frac{x_n^{p - 1}}{\rho_i{(x_n^{p - 1})}}\}}$
\end{center}
\begin{center}
 $= e_{\eta}e_{\epsilon^{-i}{\zeta^{-1}_{n}}{x_{n}}}$,
\end{center}
where $\zeta_{n}$ is a fixed $p$th root of $\zeta_{n - 1}$ and $\epsilon \neq 1$ is a $p$-th root of unity in $F$.
Replacing $x_{n}$ by $x^{-1}_{n}$ we get $e_{\rho_{i}} = e_{\eta}e_{\epsilon^{-i}{\zeta^{-1}_{n}}{x^{-1}_{n}}}$. This
proves the second statement of the theorem.
\end{proof}
\section{primitive central idempotents of an abelian $p$-group over $\mathbb{Q}$}
In this section, we compute the primitive central idempotents of an abelian $p$-group algebra over $\mathbb{Q}$. We shall proceed inductively. Let $G$ be an abelian $p$-group. Then $G$ always have a composition series. For computing primitive central idempotents of $\mathbb{Q}[G]$ inductively we choose a composition series of $G$ in such a way that $i$ is increasing as well as  $j$ is decreasing in the equation $(2.1)$.


\subsection{Rational Group Algebra of a Cyclic Group}
If $F$ and $E$ are fields, $F \subseteq E$, we denote by $[E : F]$ the degree of the field extension; that is, the dimension of $E$
as a vector space over $F$.\\
An element $\zeta \in \mathbb{C}$ is said to be a root of unity if $\zeta^{n} = 1$ for some $n \geq 1$. A root of unity $\zeta$ is said to be
primitive $m$th root of unity if $m$ is the smallest positive integer with the property that $\zeta^{m} = 1$. The $m$th cyclotomic polynomial
is
\begin{center}
 $\Phi_{m}(x) = \displaystyle{\prod_{\zeta}(x - \zeta)}$,
\end{center}
where the product has taken over all primitive $m$th roots of unity. It is well known that $\Phi_{m}(x)$ has coefficients in $\mathbb{Q}$ and this
polynomial is irreducible over $\mathbb{Q}$. The degree of the the polynomial $\Phi_{m}(x)$ is $\phi(m)$, where $\phi$ is the Euler $\phi$
function; that is, $\phi(m)$ is the number of integers $k$, $1 \leq k \leq m$, which are relatively prime to $m$.\\
Let $\zeta_{m}$ be a primitive $m$th root of unity in $\mathbb{C}$. Then the mapping
\begin{center}
 $\frac{\mathbb{Q}[X]}{\Phi_{m}(X)} \longrightarrow \mathbb{Q}(\zeta_{m})$
\end{center}
defined by $x + (\Phi_{m}(x)) \longmapsto \zeta_{m}$ is an isomorphism and so
\begin{center}
 $\mathbb{Q}(\zeta_{m}) \cong \frac{\mathbb{Q}[x]}{(\Phi_{m}(x))}$
\end{center}
is a field extension of $\mathbb{Q}$ of degree $\phi(m)$ known as a cyclotomic field. Since for $n \geq 1$, the polynomial $x^{n} - 1$
factors
\begin{center}
 $x^{n} - 1 = \displaystyle{\prod_{m \mid n}}{\Phi_{m}(x)}$
\end{center}
the rational group algebra of the cyclic group $C_{n}$ of order $n$ is
\begin{center}
 $\mathbb{Q}[C_{n}] \cong \frac{\mathbb{Q}[x]}{(x^{n} - 1)} \cong \displaystyle{\bigoplus_{m \mid n}{\frac{\mathbb{Q}[x]}{\Phi_{m}(x)}}} \cong
 \bigoplus_{m \mid n}{\mathbb{Q}(\zeta_{m})}$.
\end{center}
\subsection{Primitive Central Idempotents of $C_{p^{n}}$ over $\mathbb{Q}$}
The following proposition gives an explicit description of the complete set of primitive central idempotents of $C_{p^{n}}$ over $\mathbb{Q}$ using long presentation.
\begin{proposition}
Let $G$ be $C_{p^{n}}$, where $n$ is a positive integer. Then $G$ has the long presentation:
\begin{center}
$G = \langle{x_{1}, x_{2}, \dots, x_{n} \mid x^{p}_{1} = 1, x^{p}_{2} = x_{1}, \dots , x^{p}_{n} = x_{n - 1}}\rangle$.
\end{center}
Let $e_{0} = e_{x_{1}}e_{x_{2}} \dots e_{x_n}$ and $e_{i} = e_{x_{1}}e_{x_{2}} \dots e_{x_{i - 1}}e^{'}_{x_{i}}$, where $i \in \{1, 2, \dots , n\}$. Then the set $\{e_{0}, e_{1}, \dots , e_{n}\}$ is the complete set of primitive central idempotents of $\mathbb{Q}[G]$.
\end{proposition}
\begin{proof}
Since the rational group algebra $\mathbb{Q}[G] \cong \oplus_{i = 0}^{n}{\mathbb{Q}(\zeta_{p^{i}})}$, then $\mathbb{Q}[G]$ contains
precisely $n + 1$ primitive central idempotents.
One can show that
$(e_{x_{1}}e_{x_{2}} \dots e_{x_{i}})(e_{x_{1}}e_{x_{2}} \dots e_{x_{j}}) = (e_{x_{1}}e_{x_{2}} \dots e_{x_{j}})$, for $1 \leq i \leq j$, and therefore $e_{i}$'s are idempotents. Now we verify orthogonality condition. 
For $i \geq 1, e_{0}e_{i} = (e_{x_{1}}e_{x_{2}} \dots e_{x_n})e_{i}
= (e_{x_{1}}e_{x_{2}} \dots e_{x_n})\{(e_{x_{1}}e_{x_{2}} \dots e_{x_{i - 1}}) - (e_{x_{1}}e_{x_{2}} \dots e_{x_{i}})\}
= 0.$
Hence $e_{0}$ is orthogonal to $e_{i}$.
Now assume that $1 \leq i \lneq j$, then
\begin{align*}
{e_i}{e_{j}} &= (e_{x_{1}}e_{x_{2}} \dots e^{'}_{x_{i}})(e_{x_{1}}e_{x_{2}} \dots e^{'}_{x_{j}})\\
&= \{e_{x_{1}}e_{x_{2}} \dots e_{x_{i - 1}}(1 - e_{x_{i}})\}\{e_{x_{1}}e_{x_{2}} \dots e_{x_{j - 1}}(1 - e_{x_{j}})\}\\
& = 0.
\end{align*}
It is easy to see that $\sum_{i = 0 }^{n}{e_{i}} = 1$. It follows that $e_{0}, e_{1}, \dots , e_{n}$ are
$n + 1$ pairwise orthogonal central idempotents whose sum is 1, hence this is the complete set of primitive central idempotents of
$\mathbb{Q}[G]$. This completes the proof of the proposition.
\end{proof}

\subsection{PCI-diagram of an Abelian $p$-Group over $\mathbb{Q}$}
We state four rules by which one can inductively construct primitive central idempotents and draw PCI-diagram of an abelian $p$-group over $\mathbb{Q}$.\\\\
{\textbf{Rule 0:}}
Each non trivial primitive central idempotent at any stage can be expressed in a product form and
contains exactly one $e^{'}_{z}$ as a factor, where $z$ is a generator of the chosen long presentation of $G$.\\\\
{\textbf{Proof of the Rule 0:}}
Let $e$ be a non trivial primitive central idempotent at the $l$th stage of PCI- diagram, i.e., $e$ is a non trivial primitive
central idempotent of $\mathbb{Q}[G_{l}]$. Without loss generality one can assume that $G_{l}$ is equal to $G$. By Maschake's theorem
and commutativity, the semisimple artinian ring $\mathbb{Q}[G]$ is direct sum of fields. In particular, $\mathbb{Q}[G]e$ is a field and
contains the finite subgroup $Ge = \{ge | g \in G \}$ which is necessarily cyclic, say of order $p^{n}$, $n \geq 0$.
Let $K = \{ g \in G \mid ge = e \}$. Then
$$Ge \cong \frac{G}{K} = \langle{\bar{x}_{1}, \bar{x}_{2}, \dots , \bar{x}_{n} \mid {\bar{x}_{1}}^{p}= 1, {\bar{x}_{2}}^{p}= \bar{x}_{1},
\dots , {\bar{x}_{n}}^{p} = \bar{x}_{n - 1}}\rangle,$$
where $\bar{x_{i}}$ denotes the image of $x_{i}$ in the factor group $G/K$. Again by lemma$5.2$, the group algebra
$\mathbb{Q}[G] = \mathbb{Q}[G]{\widehat{K}} \oplus \mathbb{Q}[G]{(1 - \widehat{K})}$ and because
$\widehat{K}e = e$, it follows that $\mathbb{Q}[G]e = \mathbb{Q}[G]{\widehat{K}}e$ is actually simple component of
$\mathbb{Q}[G]{\widehat{K}}$. Thus $e$ is also primitive central idempotent of $\mathbb{Q}[G]{\widehat{K}} \cong \mathbb{Q}[\frac{G}{K}]$.\\
Let $\phi: \mathbb{Q}[G] \rightarrow \mathbb{Q}[\frac{G}{K}]$ be the homomorphism induced by the natural group epimorphism
$G \rightarrow \frac{G}{K}$ and with the kernel $\mathbb{Q}[G]{(1 - \widehat{K})}$. Recall that $e = \widehat{K}e \longmapsto \phi(e)$ under the isomorphism
$\phi: \mathbb{Q}[G]{\widehat{K}} \rightarrow \mathbb{Q}[\frac{G}{K}]$. Thus $\phi(e)$ is a primitive central idempotent in
$\mathbb{Q}[\frac{G}{K}]$. Writing $e = \sum_{g \in G}{e(g)g}$, $e(g) \in \mathbb{Q}$, we note that some $e(g) \neq \frac{1}{\mid{G}\mid}$,
because $e$ is a non trivial primitive central idempotent of $\mathbb{Q}[G]$. Since $ke = e$ for any $k \in K$, it follows that $e(g) = e(kg)$ for any $g \in G$ and
$k \in K$, so the coefficients of $e$ is contant on cosets of $K$. Thus
\begin{center}
$e = \displaystyle{\sum_{i = 0}^{p^{n} - 1}{\alpha_{i}{\widehat{K}}{x^{i}_{n}}}}$,
\end{center}
with each $\alpha_{i} \in \mathbb{Q}$ and, for some $i$, $\frac{\alpha_{i}}{\mid{K}\mid} = \frac{1}{\mid{G}\mid}$;  that is, for some $i$,
$\alpha_{i} \neq \frac{1}{|\frac{G}{K}|}$. Since
\begin{center}
$\phi(e) = \displaystyle{\sum_{i = 0}^{p^{n} - 1}{\alpha_{i}}{{\bar{x}_{n}}^{i}}}$,
\end{center}
it must be that $\phi(e) \neq \widehat{\frac{G}{K}}$. hence by the lemma $7.1$ we get
$$\phi(e) = e_{i} = \{(e_{\bar{x_{1}}}e_{\bar{x_{2}}} \dots e_{\bar{x_{i - 1}}}) - (e_{\bar{x_{1}}}e_{\bar{x_{2}}} \dots e_{\bar{x_{i}}})\},$$ for some $i$, $1 \leq i \leq n$.
Then $e = \widehat{K}\{(e_{x_{1}}e_{x_{2}} \dots e_{x_{i - 1}}) - (e_{x_{1}}e_{x_{2}} \dots e_{x_{i}})\}$, which implies that
$e = \{\widehat{K}(e_{x_{1}}e_{x_{2}} \dots e_{x_{i - 1}})\}( 1 - e_{x_{i}})$, for some $i$, $1 \leq i \leq n$. Again, $\widehat{K}$ can be
expressed product of certain number of $e_{x}$, where $x \in G$. Hence each non trivial primitive central idempotent at any stage of
PCI- diagram can be expressed in a product form and contains precisely one $e^{'}_{z}$ as a factor,
where $z$ is a generator of the chosen long presentation of $G$. This completes the proof of the theorem.\\\\
\textbf{Rule 1:}
Let $e$ be the trivial primitive central idempotent in the $l$th stage of the PCI- diagram. Let $u$ be the new generator
introduced in the $l + 1$st stage of chosen long presentation of $G$. Then $e$ is adjacent to precisely two vertices
$ee_{u}$ and $ee^{'}_{u}$ in the $l + 1$st stage of PCI- diagram. In term of picture this rule is:\\
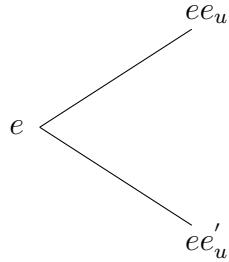
\begin{center}
\begin{tikzpicture}
\node at (0,0) {$e$};
\node at (2.5,1.5) {$ee_{u}$};
\node at (2.5,-1.5) {$ee^{'}_{u}$};
\draw (0.3,0) -- (2.3,1.3);
\draw (0.3,0) -- (2.3,-1.3);
\end{tikzpicture}
\end{center}
\vspace*{0.5cm}
\textbf{Proof of the Rule 1:}
First observe that $e$ is central idempotent in $\mathbb{Q}[G_{l + 1}]$. So $\mathbb{Q}[G_{l + 1}]e$ is an ideal in
$\mathbb{Q}[G_{l + 1}]$, which is semisimple component corresponding to the induced representation of irreducible representation
associated to $e$ from $G_l$ to $G_{l + 1}$.
Now the ideal $\mathbb{Q}[G_{l + 1}]e$ is direct sum of minimal ideals in $\mathbb{Q}[G_{l + 1}]$, in other words
direct sum of simple rings. One can see that $\mathbb{Q}[G_{l + 1}]ee_{u}$ is one simple component of $\mathbb{Q}[G_{l + 1}]e$.
Since $\mathbb{Q}[G_{l + 1}]e$ is direct sum of two simple components, then $\mathbb{Q}[G_{l + 1}]ee^{'}_{u}$ is another simple component
of $\mathbb{Q}[G_{l + 1}]e$. Hence $e$ is adjacent to the precisely two vertices $ee_{u}$ and $ee^{'}_{u}$ at the $(l + 1)$st stage of the
PCI- diagram.\\\\
\textbf{Rule 2:}
Let $e$ be a non trivial primitive central idempotent in the $l$th stage of PCI- diagram. Then $e$ contains precisely one
$e^{'}_{z}$ as a factor, where $z$ is a generator of the chosen long presentation. Let $u$ be the long generator introduced in the
$l + 1$st stage of PCI- diagram such that $z = u^{p^s}$, for some +ve integer $s$. Then $e$ is adjacent to itself in the $l + 1$st stage of
the PCI- diagram. In term of picture this rule is:\\
 \begin{center}
\begin{tikzpicture}
\node at (0,0) {$e$};
\node at (2.5,0) {$e$};
\draw (0.3,0) -- (2.4,0);
\end{tikzpicture}
\end{center}
\vspace{0.5cm}
\textbf{Proof of the Rule 2:}
Since $e$ is a non trivial primitive central idempotent at the $l$th stage of the PCI- diagram, then $e$ is a  non trivial
primitive central idempotent of the group algebra $\mathbb{Q}[G_{l}]$. By previous theorem, $e$ is pull back of primitive central idempotent
of cyclic quotient $G/K$, which is isomorphic to $C_{p^{n}}$(say), for some $n \geq 1$, infact it is pull back of unique faithfull irreducible
representation of $G/K$. Again, since $e$ contains precisely
one $e^{'}_{z}$, where $z$ is a generator of the chosen long presentation of $G$, then $e$ is pull back of primitive central idempotent
$e^{'}_{\bar{z}}$ of ${G_l}/K$, where $\bar{z}$ si equal to the coset $zK$. Notice that $\bar{z}$ is the first generator in the long
presentation of ${G_{l}}/K$.\\
Now, as $u$ is the generator inserted in the $(l +1)$st stage of the chosen long presentation of $G$, then
$G_{l + 1} = \langle{G_{l}, u}\rangle$. Again, since $z = u^{p^{s}}$, for some +ve integer $s$, then $G_{l + 1}/K$ is isomorphic to
$C_{p^{n + 1}}$. Again notice that $\bar{z}$ is the first generator of the long presentation of $G_{l + 1}/K$. Thus $e$ is also pull back of
primitive central idempotent $e^{'}_{\bar{z}}$ of ${G_{l + 1}}/K$. Therefore $\mathbb{Q}[G_{l + 1}]e$ is minimal ideal of
$\mathbb{Q}[G_{l + 1}]$. Hence $e$ is adjacent to itself at the $(l + 1)$st stage of PCI- diagram.\\\\
\textbf{Rule 3:}
Let $e$ be a non trivial primitive central idempotent in the $l$th stage of PCI- diagram. Then $e$ contains precisely one
$e^{'}_{z}$ as a factor,
where $z$ is a generator of the chosen long presentation. Let $u$ be the long generator introduced in the $l + 1$st stage of PCI- diagram
such that $z \neq u^{p^s}$, for any +ve integer $s$. Then $e$ is adjacent to precisely $p$ vertices $ee_{u}, ee_{zu}, \dots ,
ee_{z^{p - 1}u}$ in the $l + 1$st stage of the PCI- diagram.
\newpage
In term of picture this rule is:
 
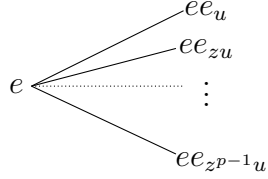
\begin{center}
\begin{tikzpicture}
\node at (0,0) {$e$};
\node at (2.5,1) {$ee_{u}$};
\node at (2.5,0.5) {$ee_{z{u}}$};
\node at (2.5,0) {$\vdots$};
\node at (2.7,-1) {$ee_{z^{p - 1}{u}}$};
%
\draw (0.2,0.0) -- (2.2, 1);
\draw (0.2,0.0) -- (2.1,0.5);
\draw [densely dotted] (0.2,0.0) to (2.2,0.0);
\draw (0.2,0.00) -- (2.1,-.9);

\end{tikzpicture}
\end{center}
\textbf{Proof of the Rule 3:}
By the theorem$(5.2)$, as $e$ is a nontrivial primitive central idempotent of $\mathbb{Q}[G_{l}]$, then $e$ is lift of primitive central
idempotent of a cyclic quotient $G/K$, $K \neq {G_{l}}$, which is siomorphic to $C_{p^{n}}$ (say), where $n$ is a positive integer. Since
$e$ contains $e^{'}_{z}$ as a factor, then first long generator of the cyclic quotient $G/K$ is equal to $\bar{z}$, where $\bar{z}$ denotes
the coset $zK$. It is quite clear that $e$ is equal to $\widehat{K}{e^{'}_{z}}$.
By hypothesis, $z \neq u^{p^s}$, for any positive integer $s$, then the subgroups:
$\displaystyle{\langle{K, u}\rangle, \langle{K, zu}\rangle, \dots , \langle{K, {z^{p - 1}}u}}\rangle$ are all distinct subgroups of $G_{l + 1}$. One can
observe that, for each $i \in \{0, 1, \dots , p - 1\}$, the quotient group ${G_{l + 1}}/ {\langle{K, {z^{i}}u}\rangle}$ is
isomorphic to ${G_{l}}/K$. So, for each $i$, the first generator of the cyclic quotient $G_{l + 1}/{\langle{K, z^{i}u}\rangle}$ is $\bar{z}$,
where $\bar{z}$ denotes the coset $z{\langle{K, z^{i}u}\rangle}$. Therefore, for ech $i$, lift of $e^{'}_{\bar{z}}$ is equal to
$\widehat{\langle{K, z^{i}u}\rangle}{e^{'}_{z}} = {e}{e_{z^{i}u}}$. Hence the set $\{e_{z^{i}u} \mid i = 0, 1, \dots , p - 1\}$ are
$p$ distinct primitive central idempotents of $\mathbb{Q}[G_{l + 1}]$. It is clear that each $\mathbb{Q}[G_{l + 1}]{ee_{z^{i}u}}$ belongs to
$\mathbb{Q}[G_{l + 1}]{e}$. One can show that the ideal $\mathbb{Q}[G_{l + 1}]{e}$ is direct sum of $p$ minimal ideals of
$\mathbb{Q}[G_{l + 1}]$ and then $\mathbb{Q}[G_{l + 1}]{e} = \displaystyle{\oplus_{i =0}^{p - 1}}\mathbb{Q}[G_{l + 1}]{ee_{z^{i}u}}$.
\begin{example}
Let\\
$G = C_{p} \times C_{p} = \langle{x_{\{(1, 2), 1\}}, x_{\{(1,1), 1\}} | x^{p}_{\{(1,2), 1\}} = 1, x^{p}_{\{(1,1), 1\}} = 1, com}\rangle.$\\\\
The associated {PCI- diagram} is
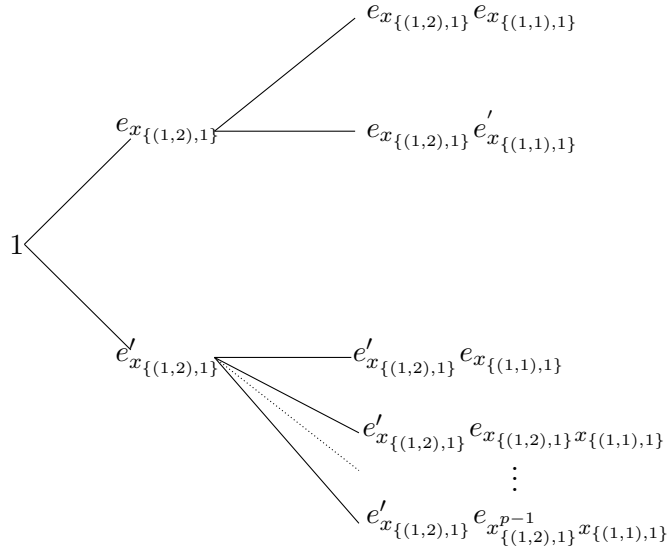
\small
\begin{center}
\begin{tikzpicture}
\node at (0,0) {$1$};
\node at (2,1.5) {$e_{x_{\{(1,2), 1\}}}$};
\node at (6,3) {$e_{x_{\{(1,2), 1\}}}e_{x_{\{(1,1), 1\}}}$};
\node at (6,1.5) {$e_{x_{\{(1,2), 1\}}}e^{'}_{x_{\{(1,1), 1\}}}$};
\node at (2,-1.5) {$e'_{x_{\{(1,2), 1\}}}$};
\node at (5.82,-1.5) {$e'_{x_{\{(1,2), 1\}}}e_{x_{\{(1,1), 1\}}}$};
\node at (6.55,-2.5) {$e'_{x_{\{(1,2), 1\}}}e_{{x_{\{(1,2), 1\}}}{x_{\{(1,1), 1\}}}}$};
\node at (6.55,-3) {$\vdots$};
\node at (6.60,-3.7) {$e'_{x_{\{(1,2), 1\}}}e_{{x^{p - 1}_{\{(1,2), 1\}}}{x_{\{(1,1), 1\}}}}$};
%

%
\draw (2.6,1.5) to (4.45,1.5);
\draw (2.6,1.5) to (4.45, 3);
\draw (2.6,-1.5) to (4.4,-1.5);
\draw (2.6,-1.5) to (4.5,-2.5);
\draw (2.6,-1.5) to (4.5,-3.7);
\draw (0.1, 0) -- (1.5,1.4);
\draw (0.1, 0) -- (1.5,-1.4);
\draw [densely dotted](2.6, -1.5) to (4.5, -3);
\end{tikzpicture}
\end{center}
\end{example}

\section{Wedderburn Decomposition of Rational Group Algebra of an Abelian $p$-Group}
In this section we compute the coefficients of cyclotomic fields, which are appearing in the Wdderburn decomposition of rational group algebra of an abelian $p$-group. Let us fix some notations for that.\\\\
Let $G = \prod_{r}G_{r}$,
where $G_{r}$ is a product of $a_{r}$ copies of cyclic groups of order $p^{r}$, and of exponent $p^{n}$. Let $b_{r} = a_{r} + a_{r + 1} + \dots + a_{n}$ and
$c_{r} = a_{1} + 2a_{2} + \dots + (r-1)a_{r-1}$. Let $B_{r} = \prod_{s<r}{G_{s}}$ and then $G = B_{r} \times \prod_{s \geq r}{G_{s}}$. Let
$\Omega^{r}(G) = \{g \in G | g^{p^r} = 1\}$, then $\Omega^{r}(G) = B_{r} \times$ {product of $b_{r}$ copies of cyclic groups of order $p^{r}$}.
Let $H_{r}$ be the product of $b_{r}$ copies of cyclic groups of order $p^{r}$. Then $\Omega^{r}(G) = B_{r} \times H_{r}$. Let $E_{r}(G) = \{g \in G | o(g) = p^{r}\}$, then $|E_{r}(G)| = |\Omega^{r}(G)| -|\Omega^{r-1}(G)|$. Therefore the number of cyclic subgroups of $G$ isomorphic to
$C_{p^{r}}$ is equal to ${|E_{r}(G)|}/{\phi(p^{r})}$, where $\phi$ denotes Euler's phi function. The next proposition says that the number
${|E_{r}(G)|}/{\phi(p^{r})}$ is same as the number of subgroups of $G$ with factor groups isomorphic to
$C_{p^{r}}$.
\begin{proposition}
The number of subgroups of $G$ with factor groups isomorphic to
 $C_{p^r}$ is equal to the number of cyclic subgroups of $G$ isomorphic to $C_{p^r}$, and therefore this number is same as ${|E_{r}(G)|}/{\phi(p^{r})}$.
\end{proposition}
\begin{proof}
Let $H$ be a subgroup of $G$, such that  $G/H$ isomorphic to $C_{p^r}$.
We consider ${\widetilde{H}} = \{\chi \in \hat{G}| \chi = 1 ~on~ H\}$, then ${\widetilde{H}}$ is a
subgroup of $\hat{G}$.
Let $\psi$ be an isomorphism from $\hat{G}$ onto
$G$, and assume that $\psi({\widetilde{H}})$ is equal to $K$(say). It is easy to show that
${\widetilde{H}}$ is isomorphic to $C_{p^r}$, and hence $K$ is isomorphic to $C_{p^r}$.\\

Suppose  that $H_1, H_2$ are two distinct subgroups of $G$ such that factor groups are ispmorphic to $C_{p^r}$, and then there exists an element $h \in {H_1 - H_2}$,
and therefore $\widetilde{H_1} \neq \widetilde{H_2}$. So $\psi{\widetilde{H_1}} \neq \psi{\widetilde{H_2}}$.
So the number of subgroups of $G$ with factor groups are isomorphic to $C_{p^r}$ is equal to the number of subgroups $G$ isomorphic to $C_{p^r}$, and hence this number is equal to ${|E_{r}(G)|}/{\phi(p^{r})}$.
\end{proof}

\begin{theorem}
 Let $G$ be an abelian $p$-group, of exponent $p^n$. For $r \leq n$, the coefficient of $\mathbb{Q}(\zeta_{p^{r}})$ in the Wedderburn decomposition of $\mathbb{Q}[G]$ is a polynomial in $p$. In fact it is equal to $p^{c_{r} + (r -1 )b_{r - 1}}(\frac{p^{{b_{r}}} - 1}{p - 1})$.
\end{theorem}
\begin{proof}
 By the theorem (see ), the coefficient of $\mathbb{Q}(\zeta_{p^{r}})$ in the Wedderburn decomposition of $\mathbb{Q}[G]$ is equal to
 $\frac{|E_{r}(G)|}{\phi(p^{r})}$, where $\phi$ denotes Euler's phi function. But $|E_{r}(G)| = |B_{r}||E_{r}(H_{r})| = p^{c_{r}}(p^{r{b_{r}}}  - p^{(r - 1)b_{r}}) = p^{c_{r} + {(r - 1)b_{r}}}(p^{b_{r}} - 1)$. Thus the coefficient of $\mathbb{Q}(\zeta_{p^{r}})$ in the Wedderburn decomposition of $\mathbb{Q}[G]$ is equal to $p^{c_{r} + {(r - 1)b_{r}}}(p^{b_{r}} - 1)/(p^{r} - p^{r-1}) = p^{c_{r} + (r -1 ){(b_{r}- 1})}(\frac{p^{{b_{r}}} - 1}{p - 1})$.
\end{proof}
\begin{remark}
By the above theorem the coefficient of $\mathbb{Q}(\zeta_{p^{r}})$ in the Wedderburn decomposition of $\mathbb{Q}[G]$ is equal to $p^{c_r + (r-1)(b_{r}-1)}\{1 + p + \dots + p^{(b_r -1)}\}$. One can observe that first part of this polynomial is power of $p$, on the other hand the second part is coprime to $p$.
\end{remark}

\end{document}